\documentclass[letterpaper, 10 pt, conference]{ieeeconf}  

\IEEEoverridecommandlockouts                              
\title{\LARGE \bf
Nonlinear Dual control based on Fast Moving Horizon estimation and Model Predictive Control with an observability constraint
}

\usepackage[autostyle]{csquotes}
\usepackage{amsmath,verbatim,amssymb,amsfonts,amscd, graphicx}
\usepackage{mathabx}
\usepackage{breqn}
\usepackage{float}
\usepackage{graphics}
\usepackage{stmaryrd} 
\usepackage{textcomp}
\usepackage{setspace}
\usepackage{algorithm}
\usepackage{algpseudocode}
\usepackage{optidef}

\author{Emilien Flayac$^{1}$,  Girish Nair$^{1}$, Iman Shames $^{2}$, 
    \thanks{$^{1}$ Emilien Flayac (emilien.flayac@unimelb.edu.au) and Girish Nair (girish.nair@unimelb.edu.au) are with the Department of Electrical \& Electronic Engineering, University of Melbourne, Melbourne, Australia }%
    \thanks{$^{2}$Iman Shames (iman.shames@anu.edu.au) is with the  School of Engineering, The Australian National University, Canberra, Australia.}%
    \thanks{This work received funding from the Australian Government, via grant AUSMURIB000001 associated with ONR MURI grant N00014-19-1-2571}
}

\newtheorem{proposition}{Proposition}[section]
\newtheorem{definition}{Definition}[section]
\newtheorem{assumption}{Assumption}[section]
\newtheorem{theorem}[proposition]{Theorem}
\newtheorem{corollary}[proposition]{Corollary}
\newtheorem{lemma}[proposition]{Lemma}
\newtheorem{remark}{Remark}[section]

\begin{document}

\allowdisplaybreaks

\maketitle
\thispagestyle{empty}
\pagestyle{empty}

\begin{abstract}
This paper proposes an algorithm that combines Fast Moving Horizon Parameter Estimation and Model Predictive Control subject to an observability constraint designed to ensure a lower bound on the performance of the parameter estimator. Output-feedback stability is proved through input-to-state stability of the state/error system under a small noise and initial error assumption. Numerical experiments have been carried out in the case of Active Simultaneous Localisation and Mapping (SLAM).  
\end{abstract}

\section{INTRODUCTION}
Optimisation-based estimation techniques like Moving Horizon Estimation (MHE), in which one tries to recover the trajectory of a system through solving an optimisation problem, are arousing growing interests in both their theory \cite{muller_nonlinear_2017}, \cite{rawlings_optimization-based_2012} and their practical implementations \cite{diehl_efficient_2009}.  System identification is benefiting from these advances too, see \cite{kuhl_real-time_2011} for example. The main advantage of MHE is that it can deal with nonlinear systems and constraints while also aiming for tractability. Indeed, only the measurements coming from a time window of fixed size are used at each time step. A variation of MHE, called Fast  MHE, involves solving the associated optimisation problem partially. One only performs a few iterations of a dedicated optimisation routine at each time step and uses the current estimate as the initial guess for the next problem. Much work has been done in this direction, see \cite{alessandri_moving-horizon_2016}, \cite{alessandri_fast_2017}, \cite{kang_moving_2006}, \cite{wynn_convergence_2014} and \cite{zavala_stability_2010}. In all this work, the MHE scheme is enabled by the so-called \emph{$N$-step observability} condition. It states that a small output error on a rolling time window must lead to a small estimation error uniformly with respect to time if one starts sufficiently close to the reference. These conditions are also often assumed to hold uniformly with respect to the input applied to the system. However, in the general nonlinear case, the input might have an influence on the observability condition and thus the quality of the estimation process.  In adaptive control, where one tries to regulate a system while also identifying its dynamical model, this phenomenon has been well known since the seminal work of Feldbaum \cite{feldbaum_theory_1960}.  He stated that adaptive controllers must be designed  to guide the system in a standard way, as well as to probe information and excite the system, which leads to good parameter estimation. This is known as the \emph{dual effect} of the controller. See \cite{mesbah_stochastic_2017} for a survey. It implies that the \emph{separation principle}, which claims that independent design of the control and estimation schemes can be efficient, does not hold in a general nonlinear case.  In the context of Moving Horizon Estimation, some work has been done to combine it with Model Predictive Control (MPC). In the manner of MHE, MPC consists in solving a finite-horizon optimal control problem on a rolling basis. In \cite{copp_nonlinear_2014}, a minimax MHE-MPC output feedback scheme is presented  although the observability assumption is assumed to hold independently of the feedback scheme. In \cite{ellis_robust_2014}, the authors successfully mix an Economic MPC scheme with a specific MHE technique that requires a high-gain observer under feedback. However, the construction of such a feedback is not specified in general and seems non-trivial. In \cite{bruggemann_model_2020},  a MPC with a Persistence of Excitation condition is combined with a recursive parameter estimation algorithm. Still, the influence of the parameter on the dynamics is supposed to be affine and a periodic persistent state trajectory is also required. That is why, in this paper, we present a dual output-feedback scheme for nonlinear discrete-time systems with noisy measurements, based on a fast MHE algorithm for parameter estimation and on an MPC with a constraint on the predicted \emph{Observability Grammian}. The design of the controller and the estimator are coupled in the sense that, at each time step, the MPC algorithm aims to ensure good estimation performance at the following step through an appropriate enforced observability condition. Contrary to \cite{bruggemann_model_2020}, our method does not require periodicity or any explicit property as the required excitation coming from the input is generated by a general implicit constraint.     

The remainder of the paper is structured as follows. Section \ref{section_notation} gathers standard notations that will be used the entire paper. Section \ref{section_setup} describes the setup of a nonlinear parametrised discrete-time system with noisy measurements. Section \ref{section_MHPE} presents a typical Moving Horizon Parameter Estimation Problem and its fast implementation. In Section \ref{section:obs_control}, our dual MPC scheme with an observability constraint is set. In Section \ref{section:output-feedback_stab}, output-feedback stability is established in terms of input-to-state stability of the state/error system  under a small noise and initial error assumption. Finally, our estimation and control scheme is applied in Section \ref{section:application} to the problem of Active Simultaneous Localisation and Mapping (SLAM).

\section{NOTATIONS}\label{section_notation}
We respectively denote by $\mathbb{N}$ and $\mathbb{N}^*$ the set of non-negative and positive integers.  For some $n\in\mathbb{N}^*$ and $m\in\mathbb{N}^*$  and $x \in \mathbb{R}^n$, $\Vert x\Vert$ denotes the Euclidian norm of $x$. For  $x \in \mathbb{R}^n$ and $R>0$,  $\widebar{B}(x,R)$ denotes the closed ball of radius $R$ centered at $x$. For $L\in\mathbb{N}^* $ and a finite sequence of vectors $X\in {(\mathbb{R}^n)}^L$, $\Vert X\Vert=\max(\Vert X_1 \Vert,\dots, \Vert X_L \Vert)$. For a bounded infinite sequence of vectors $X\in(\mathbb{R}^n)^{\mathbb{N}}$, $\Vert X \Vert_{\infty}$ denotes its $\ell_{\infty}$-norm.  For a linear operator $A$ from $\mathbb{R}^n$ to $\mathbb{R}^m$, $\Vert A \Vert$ denotes the operator norm induced by the Euclidian norm. For $i\in \mathbb{N}^*$, a function $F:\mathbb{R}^n\rightarrow \mathbb{R}^m$ is called a $\mathcal{C}^i$ function if it is $i^{th}$-times continuously differentiable. Its $i^{th}$-differential is denoted by $\nabla^{i}F$ with the simplification $\nabla^{1}F=\nabla F$. In the sequel, $\succeq$ denotes the partial order on positive semi-definite matrices and $I$ denotes the identity matrix.  A function $\alpha: \mathbb{R}^+\rightarrow \mathbb{R}^+$ is called a $\mathcal{K}$-function if it is continuous, increasing and such that $\alpha(0)=0$. It is called a $\mathcal{K}_\infty$-function if it is also unbounded. A function $\beta: \mathbb{R}^+\times\mathbb{R}^+\rightarrow \mathbb{R}^+$ is called a $\mathcal{KL}$-function if for any $t\geq0$, $\beta(\cdot,t)$ is a  $\mathcal{K}$-function and if for any $r\geq0$, $\beta(r,\cdot)$ is decreasing and converges to $0$ at infinity.

\section{PROBLEM SETUP} \label{section_setup}
Consider the following nonlinear discrete-time partially observed dynamical system, defined for any $t\geq0$:
\begin{align}
    &x_{t+1}=f(x_t,u_t,\bar{p}),\label{eq:gen_dyn_sys}
    &y_t^0=\begin{bmatrix} x_t^0 \\y_t\end{bmatrix},
\end{align}
where $x_0\in \mathbb{R}^{n_x}$ is given,  $x_t^0=x_t+v^0_t$ and $y_t=h(x_t,\bar{p})+v_t$. Moreover, $x_t$ is a state variable valued in $\mathbb{R}^{n_x}$, $\bar{p}$ is an unknown parameter valued in a given set $\mathcal{P}\subset \mathbb{R}^{n_p}$ and  $u_t$ is a control variable valued in a given set $\mathcal{U}\subset \mathbb{R}^{n_u}$. Note that $x_t^0$ is a measurement of the state, $y_t$ is an observation valued in $\mathbb{R}^{n_y}$ involving also the parameter $\bar{p}$ and $v_t$ and $v_t^0$ are unmodelled bounded disturbances on the observations respectively valued in $\mathbb{R}^{n_y}$ and $\mathbb{R}^{n_x}$. The  dynamics of the system $f:\mathbb{R}^{n_x}\times \mathbb{R}^{n_u}\times  \mathbb{R}^{n_p}\rightarrow \mathbb{R}^{n_x} $  is such that $f(0,0,\bar{p})=0$ and $h:\mathbb{R}^{n_x}\times  \mathbb{R}^{n_p}\rightarrow \mathbb{R}^{n_y}$ is the observation function.

\begin{assumption}\label{as:regularity+compact_control_set}
Functions $f$ and $h$ are $\mathcal{C}^3$ functions, $\mathcal{U}$ is a compact set and  $\mathcal{P}$ is a closed convex set.
\end{assumption}

\begin{assumption}\label{as:bounded_noise}
 There exists $\nu>0$, such that  for any $t \geq 0$, $\Vert v_t\Vert\leq\nu$ and , $\Vert v_t^0\Vert\leq\nu$.
\end{assumption} 
In the rest of the paper, we assume that Assumption  \ref{as:bounded_noise} holds and $\nu$ is given.
\section{MOVING HORIZON PARAMETER ESTIMATION }\label{section_MHPE}
\subsection{Setup}
In this section, we present a perturbed and unperturbed Moving Horizon Parameter Estimation (MHPE) and show in Proposition \ref{prop:existence_uniqueness} the existence and uniqueness of the solutions of these problems as well as a bound between their optimiser that holds under an observability condition. Fix $L\in \mathbb{N}^*$. In the following, for any $t \geq 0$, any disturbance sequence $v_{t,L}=(v_{t-L+1},\dots,v_t) \in (\mathbb{R}^{n_y})^L$, any control sequence $u_{t,L}=(u_{t-L+1},\dots,u_{t-1}) \in (\mathbb{R}^{n_u})^{L-1}$, any starting state $x_{t-T+1} \in \mathbb{R}^{n_x}$, and any parameter $p\in \mathbb{R}^{n_p}$, we define the  cumulative output error, $C$ as follows: 
\begin{align*}
    C(p,x_{t-L+1},u_{t,L},v_{t,L})=\sum_{k=t-L+1}^t\Vert y_k-h(x_k,p) \Vert^2+\theta(p),
\end{align*}
where $\theta$ is a penalty function and the state sequence $x_{t,L}=(x_{t-T+1},\dots,x_t)$  follows the dynamics \eqref{eq:gen_dyn_sys} with input $u_{t,L}$ and initial condition $x_{t-L+1}$. We also define the noise-free output error, $\widebar{C}$ as follows:
\begin{align*}
\widebar{C}(p,x_{t-L+1},u_{t,L})&= C(p,x_{t-L+1},u_{t,L},0),\\
                                             &=\sum_{k=t-L+1}^t\Vert \bar{y}_k-h(x_k,p) \Vert^2+\theta(p),
\end{align*}
where $\bar{y}_k=h(x_k,\bar{p})$.

Thus, the MHPE problem can be defined as follows:
\begin{equation}
\begin{array}{rrclcc}
 C^*(x_{t-L+1},u_{t,L},v_{t,L})=\displaystyle \underset{p\in \mathbb{R}^{n_p}}{\text{min}} & \multicolumn{3}{l}{  C(p,x_{t-L+1},u_{t,L},v_{t,L})} \\
\textrm{s.t.} & x_{k+1} = f(x_{k},u_{k},p)
\end{array}
\label{eq:noisy_MHE}
\end{equation}
as well as the noise-free MHPE:
\begin{equation}
\begin{array}{rrclcc}
\widebar{C}^*(x_{t-L+1},u_{t,L})=\displaystyle \underset{p\in \mathbb{R}^{n_p}}{\text{min}} & \multicolumn{3}{l}{   \widebar{C}(p,x_{t-L+1},u_{t,L})} \\

\textrm{s.t.} & x_{k+1} = f(x_{k},u_{k},p)
\end{array}
\label{eq:noise_free_MHE}
\end{equation}
\begin{assumption}\label{as:penalisation}
Function $\theta$ is a non-negative smooth convex penalty function such that $\theta(p)=0$ for any $ p\in \mathcal{P}$.  
\end{assumption}

\begin{lemma}

Under Assumption \ref{as:regularity+compact_control_set},  for any $v_{t,L} \in (\mathbb{R}^{n_y})^L$, $u_{t,L} \in (\mathbb{R}^{n_u})^L$, and $x_{t-T+1} \in \mathbb{R}^{n_x}$, 
\begin{align*}
\widebar{C}^*(x_{t-L+1},u_{t,L})&=\widebar{C}(\bar{p},x_{t-L+1},u_{t,L})=0,\\
\nabla_p\widebar{C}(\bar{p},x_{t-L+1},u_{t,L})&=0,                                       
\end{align*}
and $C$ and $\widebar{C}$ are $\mathcal{C}^3$ functions. In particular, $\bar{p}$ is a global minimiser of Problem \eqref{eq:noise_free_MHE}. 
\end{lemma}
\begin{proof}
The first item follows from the definition and the non-negativity of $\bar{C}$ and the second one from the regularity assumptions.  
\end{proof}

\begin{proposition}\label{prop:diff_C}
Under Assumptions \ref{as:regularity+compact_control_set} and \ref{as:penalisation},  for any $v_{t,L} \in (\mathbb{R}^{n_y})^L$, $u_{t,L} \in (\mathbb{R}^{n_u})^L$, $x_{t-T+1} \in \mathbb{R}^{n_x}$, and $p\in\mathbb{R}^{n_p}$ 
\begin{align*}
    \nabla_p^2 C(p,x_{t-L+1},u_{t,L},v_{t,L})&= \mathcal{O}(p,x_{t-L+1},u_{t,L})+\nabla^2\theta(p)\\&+\mathcal{R}(p,x_{t-L+1},u_{t,L},v_{t,L}),\\
    \nabla_p^2 \bar{C}(\bar{p},x_{t-L+1},u_{t,L})&=\mathcal{O}(\bar{p},x_{t-L+1},u_{t,L}) +\nabla^2\theta(\bar{p}),
\end{align*}
where $\mathcal{O}=NN^T$ is called the \emph{Observability Grammian} with $N$ depending only on first order terms with respect to p and where $\mathcal{R}(p,x_{t-L+1},u_{t,L},v_{t,L})$ is a second order term with respect to p. 
\end{proposition}
\begin{proof}
This can be derived from classical online Moving Horizon Estimation results. For example see \cite{diehl_efficient_2009}
\end{proof}

We now show the existence and uniqueness of solution of Problem \eqref{eq:noisy_MHE}. We first need to introduce an assumption on the Observability Grammian at $\bar{p}$.

\begin{definition}[L-step observability after feedback]\label{as:observability_feedback} 
For  $M>0$, $\bar{\delta}>0$ and $t\geq L-1$, system \eqref{eq:gen_dyn_sys} is said to be \emph{L-step observable after feedback} on $\widebar{B}(0,M)$ at time $t$ with level $\bar{\delta}$,  iff  $x_{t-L+1}\in \widebar{B}(0,M)$ implies that there exists a control sequence  $u_{t,L} \in {\mathcal{U}}^L$ that satisfies:
\begin{align}
   \mathcal{O}(\bar{p},x_{t-L+1},u_{t,L})\succeq \bar{\delta} I. \label{eq:obs_gram}
\end{align}

\end{definition}
\begin{proposition}[Local existence and uniqueness]
\label{prop:existence_uniqueness}
  For  $M>0$ and $\bar{\delta} >0$, assume that system \eqref{eq:gen_dyn_sys} is $L$-step observable after feedback on $\widebar{B}(0,M)$ at time $t$ with level $\bar{\delta}$. Under Assumptions \ref{as:regularity+compact_control_set}, \ref{as:bounded_noise} \ref{as:penalisation} there exist ${K}(M,\nu)\geq0$ that is non decreasing with respect to  $\nu$ such that if:
  \begin{align}
      K(M,\nu)\nu< \bar{\delta} \label{eq:neighhood_eps}
  \end{align}
 then for any $x_{t-L+1}\in \widebar{B}(0,M)$ and $u_{t,L} \in {\mathcal{U}}^L$ satisfying Equation \eqref{eq:obs_gram} and for any $\Vert v_{t,L}\Vert \leq \nu$, $C(\cdot,x_{t-L+1},u_{t,L},v_{t,L})$ admits a unique minimiser  on $\widebar{B}(\bar{p},\nu)$ denoted by  $p^*_{t,\nu}$.  Moreover, $p^*_{t,\nu}$ is a strict local minimiser of Problem \eqref{eq:noisy_MHE}.  

\end{proposition}

\begin{proof}
The result follows from \eqref{eq:obs_gram} and the local Lipschitz property of $\mathcal{O}$, $\mathcal{R}$ and $\theta$ around $(x_{t-L+1},\bar{p}, v_{t,L})$

\end{proof}

The existence and uniqueness result is standard,  see \cite{fiacco_introduction_1983}. The major point of Proposition \ref{prop:existence_uniqueness} is that the existence and uniqueness result that holds under \eqref{eq:neighhood_eps} is time invariant. Indeed, Proposition \ref{prop:existence_uniqueness} means that if the noise on the measurements is sufficient small, Problem \eqref{eq:noisy_MHE} has a locally unique solution  that is close to $\bar{p}$ uniformly with respect to  $x_{t-L+1}$, $u_{t,L}$ and $v_{t,L}$. In the following, $p^*_{t,\nu}$ will denoted by $p^*_t$ when it exists.


\subsection{Fast estimation algorithm}
Classical MHE algorithms aim to compute $p^*_t$ precisely which may require to perform numerous iterations of some optimisation routine, $\psi_t:\mathbb{R}^{n_p}\rightarrow \mathbb{R}^{n_p}$. As a consequence, cheap techniques where one
computes an estimate of $p^*_{t}$ denoted by $\hat{p}_t$  by performing a few iterations of the optimisation routine have been introduced. See \cite{diehl_efficient_2009} for a review on fast MHE. Therefore, we make the following assumption:

\begin{assumption}[Optimisation algorithm]\label{as:opti_algo}\hfill

 There exist $\bar{r}>0$, $\gamma>0$ such that for any $t\geq L-1$, if $p^*_{t}$ exists then for any $p \in \widebar{B}(\bar{p},r)$:
  \begin{align}
      \Vert\psi_t(p)-p^*_{t} \Vert\leq \gamma\Vert p-p_t^* \Vert \label{eq:opti_algo_convergence}
  \end{align}
\end{assumption}
Under Assumption \ref{as:opti_algo}, for any $t\geq L-1$ and some $\hat{p}_{L-1}\in \widebar{B}(\bar{p},\bar{r})$, we define the online estimate of $\bar{p}$ as follows:
  \begin{align}
      \hat{p}_{t+1}=\psi_t(\hat{p}_t).\label{eq:online_estimate}
  \end{align}


Note that $\hat{p}_t$ can be defined without requiring $p^*_{t}$ to exist.
 The goal of the following is to design a controller that satisfies Equation \eqref{eq:obs_gram} along the trajectories of \eqref{eq:gen_dyn_sys}  while also stabilising it  at $0$.

\section{OBSERVABILITY SEEKING MODEL PREDICTIVE CONTROL }\label{section:obs_control}

The controller presented in this section is composed of two parts: a controller that is  input-to-state stabilising with respect to perturbation on the control and a non-destabilising observability-seeking controller. We first recall the notion of Regional ISS stability for nonlinear discrete-time systems and then define our ISS stabilising controller.
\subsection{Regional ISS-stability for nonlinear discrete time systems}
The definitions from this section are taken from \cite{limon_input--state_2009}.
For  $(n,m)\in (\mathbb{N}^*)^2$, consider a system of the following form for $t\geq0$:
\begin{align}
    \chi_{t+1}=g_t(\chi_t,w_t), \label{eq:gen_ISS_sys}
\end{align}
where $\chi_t\in \mathbb{R}^{n}$ is the state and $w_t \in \mathcal{W}\subset \mathbb{R}^{m} $ is a disturbance such that  $0\in \mathcal{W}$ and $g_t:\mathbb{R}^n\times\mathbb{R}^m\rightarrow\mathbb{R}^n$ satisfies $g(0,0)=0$.

\begin{definition}[Robustly  positively invariant (RPI) set]
Let $\Omega\subset \mathbb{R}^n$. The set $\Omega$ is called a \emph{(uniformly) Robustly Positively Invariant} set for system \eqref{eq:gen_ISS_sys} iff for any $t\geq0$, $\chi \in \Omega$ and $w \in \mathcal{W}$, $g_t(\chi,w)\in\Omega $.
\end{definition}

\begin{definition}[Regional Input-to-State Stability]
Consider a RPI set $\Omega$ for System \eqref{eq:gen_ISS_sys} such that $0\in \textrm{int}(\Omega)$. System \eqref{eq:gen_ISS_sys} is said to be \emph{ Input-to-State Stable (ISS) in $\Omega$} if there exist a $\mathcal{K}\mathcal{L}$-function $\beta:\mathbb{R}^+\times\mathbb{R}^+\rightarrow\mathbb{R}^+$ and a $\mathcal{K}$-function $\gamma :\mathbb{R}^+\rightarrow\mathbb{R}^+$ such that for any bounded sequence $w\in  \mathcal{W}^{\mathbb{N}}$ and $\chi_0\in \Omega$ and any $t\geq0$:
\begin{align}
    \Vert \chi_{t}\Vert \leq \beta(\Vert \chi_0\Vert,t)+\gamma(\Vert w \Vert_{\infty})\label{eq:ISS_def_ineq}
\end{align}

If  \eqref{eq:ISS_def_ineq} holds for $\Omega=\mathbb{R}^n$, then System \eqref{eq:gen_ISS_sys} is said to be \emph{globally ISS}. 
\end{definition}

\begin{definition}[Regional ISS-Lyapunov function]\label{def:ISS_lya}
Consider a RPI set $\Omega$ for System \eqref{eq:gen_ISS_sys} such that $0\in \textrm{int}(\Omega)$. A function $V: \mathbb{R}^{n}\rightarrow  \mathbb{R}^+$ is called a \emph{ISS-Lyapunov function in $\Omega$ } if there exist two $\mathcal{K}_{\infty}$ functions $\underline{\alpha}$ and $\bar{\alpha}$ such that for any $\chi\in \Omega $:
    \begin{align*}
       \underline{\alpha}(\Vert\chi\Vert)\leq V(\chi)\leq\bar{\alpha}(\Vert\chi\Vert),
    \end{align*}
    and there exist a $\mathcal{K}_{\infty}$-function $\alpha$ and a $\mathcal{K}$-function $\sigma$ such that for any $t\geq0$, $\chi_0\in \Omega$, one has along the trajectories of system \eqref{eq:gen_ISS_sys}:
    \begin{align*}
        V(\chi_{t+1})-V(\chi_t)\leq-\alpha(\Vert\chi_t\Vert)+\sigma(\Vert w_t\Vert).
    \end{align*}

If these items hold for $\Omega=\mathbb{R}^n$, $V$ is called a \emph{global ISS-Lyapunov function}. 
\end{definition}

\begin{proposition}\label{prop:ISS}
For any RPI set $\Omega$, if system \eqref{eq:gen_ISS_sys} admits a ISS-Lyapunov function in $\Omega$ then it is ISS in $\Omega$.
\end{proposition}
\begin{proof}
See \cite{limon_input--state_2009}
\end{proof}

In the sequel, we assume that there exists a feedback controller that makes system \eqref{eq:gen_dyn_sys} globally ISS with respect to  disturbances in the input with the nominal parameter $\bar{p}$.
\begin{assumption}\label{as:ISS_nominal_state_feedback}
There exist a globally Lipschitz nominal state-feedback $\kappa:\mathbb{R}^{n_x}\times\mathbb{R}^{n_p}\rightarrow \mathcal{U}$ with Lipschitz constant $L_{\kappa}$ such that $\kappa(0,\bar{p})=0$  and a function $V: \mathbb{R}^{n_{x}}\rightarrow  \mathbb{R}^+$ such that the system defined for any $t\geq0$, $x_0\in \mathbb{R}^{n_x}$  and $d_t\in \mathbb{R}^{n_u}$ by:
\begin{align}
    x_{t+1}=f(x_t,\kappa(x_t,\bar{p})+d_t,\bar{p}),\label{eq:feedback_sys}
\end{align}
admits $V$ as a global ISS-Lyapunov function. 
\end{assumption}
In the rest of the paper, we assume that Assumption \ref{as:ISS_nominal_state_feedback} holds for system \eqref{eq:gen_dyn_sys}. We fix $\kappa$ and $V$ as in Assumption \ref{as:ISS_nominal_state_feedback} and the associated $\underline{\alpha}$, $\bar{\alpha}$, $\alpha$ and $\sigma$ from Definition \ref{def:ISS_lya}.

\subsection{Model Predictive Control with maximum level of observability}

 The purpose of this section is to design an MPC controller that corrects the feedback controller $\kappa$ and ensures that Assumption \ref{as:observability_feedback} is satisfied one step forward in the future while also keeping the system stable. We assume that the control horizon is one even if it means to consider augmented controls and observations. Thus, the predicted Observability Grammian $\mathcal{O}$ will be considered on $[t-L+2, t+1]$. Notice that $\mathcal{O}$ as defined in Proposition \ref{prop:diff_C}, can also be seen as a function of some state sequence $x_{t,L}$, some input sequence $u_{t,L}$ and some parameter $p$ even if  $x_{t,L}$ does not satisfy  \eqref{eq:gen_dyn_sys}. Specifically, in this section, we consider $\mathcal{O}$ at the measured states $x^0_{t,L}$. With a slight abuse of notation, $\mathcal{O}$ can be decomposed by definition as follows, for any $p\in \mathbb{R}^{n_p}$:
 \begin{align}
     \mathcal{O}(p,x^0_{t+1,L},u_{t+1,L})&=\Gamma(p,x^0_{t,L-1},u_{t,L-1})\label{eq:obs_gram_decomposition}\\
     &+S(p,x^0_{t+1,L},u_{t+1,L})\notag,
 \end{align}
where $\Gamma(p,x^0_{t,L-1})\succeq 0$ depends only on the parameter and the present and  past measured state trajectory and 
 $S(p,x^0_{t+1,L},u_{t+1,L})\succeq 0$ depends on the whole trajectory of measured states and the parameter. By omitting the dependency on the past trajectory and injecting \eqref{eq:gen_dyn_sys}, we define $S_f(p,x_{t}^0,u_t)=S(p,f(x_{t}^0,u_t,p),u_t)$.
The MPC problem can then be formulated as follows for any $t\geq L-1$ and some $0<\mu<1$, ${\delta}'>0$:

\begin{mini!}|s|[2]                
    {\delta,u_{t}^{\text{obs}}}                               
    {-\delta+c(u_t)\label{eq:cost}}   
    {\label{eq:MPC_obs}}             
    {}                                
    \addConstraint{&\kappa({x}_t^0,\hat{p}_t)+u_{t}^{\text{obs}}}{&\in \mathcal{U} \label{eq:con1}}    
    \addConstraint{&  \Vert u_{t}^{\text{obs}}\Vert\leq \sigma^{-1}&\left(\frac{\mu}{2}\alpha\left(\frac{1}{2}\Vert {x}_t^0 \Vert\right)\right)\label{eq:con2}}  
    \addConstraint{& \widehat{\Gamma}_t+ S_f(\hat{p}_t,x_{t}^0,\kappa({x}_t^0,\hat{p}_t)}{&+u_{t}^{\text{obs}})\succeq \delta I \label{eq:con3}}    %
    \addConstraint{&\delta&\geq {\delta}'\label{eq:con4}}
\end{mini!}
   where $\widehat{\Gamma}_t=\Gamma(\hat{p}_t,x^0_{t,L-1},u_{t,L-1})$ is given  and $u_{t,L-1}$ denotes the sequence of previous controls. In the following, we fix a initial control sequence $\hat{u}_{L-1,L}$. Note that Constraint \eqref{eq:con1} ensures control admissibility, Constraint \eqref{eq:con2} ensures that $u_{t}^{\text{obs}}$ does not destabilise system \eqref{eq:gen_aug_sys}, Constraint \eqref{eq:con3} ensures that  $\delta$ is lower bound on the smallest eigenvalue on the Observability Grammian, Constraint \eqref{eq:con4}  ensures that $\delta$ is bounded from below by $\bar{\delta}$ uniformly in $t$.
 The cost \eqref{eq:cost} is a combination on the lower bound on the Observability Grammian and a cost on $u_{t}^{\text{obs}}$ denoted by $c$. In the sequel, we make a feasibility assumption on Problem \eqref{eq:MPC_obs}.

\begin{assumption}\label{as:MPC_admissiblity}
For any $M>0$, there exists $ {\delta}'(M)>0$ and $0<\mu(M)<1$  such that  for any $t\geq L-1$ if $\Vert x_{t-L+2} \Vert \leq M $ then Problem \eqref{eq:MPC_obs} is feasible. 
\end{assumption}
In the rest of the paper, we suppose that Assumption \ref{as:MPC_admissiblity} holds and we denote by $\hat{u}_{t}^{\text{obs}}(x_{t,L}^0,\hat{p}_t,M)$ a feasible point of Problem \eqref{eq:MPC_obs} for any $M>0$. Finally, for $M>0$, the total control applied to system, denoted by $\hat{u}_t$ for any $t\geq L-1$,  can be written as follows:
\begin{align}
    \hat{u}_t(x_{t,L}^0,\hat{p}_t,M)=\kappa({x}_t^0,\hat{p}_t)+\hat{u}_{t}^{\text{obs}}(x_{t,L}^0,\hat{p}_t,M).\label{eq:total_control}
\end{align}

   \begin{remark}

 Note that when it is defined the sequence of input $\hat{u}_{t,L}$ satisfies by construction:
            \begin{align}
                  \mathcal{O}^0=\Gamma(\hat{p}_t,x^0_{t,L-1},\hat{u}_{t,L-1})+S_f(\hat{p}_t,x^0_t,\hat{u}_t)\succeq \delta'I 
            \end{align}
 Problem \eqref{eq:MPC_obs} is a nonlinear Semi-Definite Program because of Constraint \eqref{eq:con3} and is generally very hard to solve. However, thanks to Constraints \eqref{eq:con3} and \eqref{eq:con4}, it will be sufficient for $\hat{u}_{t}^{\text{obs}}$ to just be feasible in order to maintain observability.    
  Assumption \ref{as:MPC_admissiblity} can be seen as a one step reachability assumption of the set defined by Constraints \eqref{eq:con3} and \eqref{eq:con4} using small inputs from $\mathcal{U}$. 
  

   \end{remark}
\section{OUTPUT-FEEDBACK STABILITY}\label{section:output-feedback_stab}
For $t\geq L-1$, let $e_t=\hat{p}_t-\bar{p}$ be the parameter estimation error. Roughly speaking,  we show in this section that under the control law \eqref{eq:total_control}, both $x_t$ and $e_t$  satisfy some ISS property for small initial estimation error and small measurement noise. We first introduce the augmented system state/error. For any $t\geq L-1$, provided that  for any $L-1\leq k\leq t$, $\Vert x_{t-L+1} \Vert \leq M$ for some $M>0$ and that $\hat{u}_k$, $p^*_{k+1}$ exist we consider an augmented state $\chi_t=(x_t,e_t)$ and an augmented dynamics $g_t:\mathbb{R}^{n_x+n_p}\times\mathbb{R}^{n_x+n_p}\rightarrow\mathbb{R}^{n_x+n_p}$ such that for any $\chi_{L-1}=(x_{L-1},e_{L-1})  \in \mathbb{R}^{n_x+n_p}$, and $t\geq L-1$,  the state/error system can be written as follows:
\begin{align}
\chi_{t+1}=g_t(\chi_t,w_t),\label{eq:gen_aug_sys}
\end{align}
where $g_t(\chi_t,w_t)=\begin{bmatrix}f(x_t,\hat{u}_t(M,x_t^0),\bar{p})\\ \psi_t(\bar{p}+e_t)\end{bmatrix}$ and $w_t=(p^*_{t+1}-\bar{p},v_t^0)$. Note that the expression of $g_t(\chi_t,w_t)$ does not depend explicitly of $w_t$. However, $w_t$ is a disturbance term that will be used to represent the effect of the measurement noise in the right-hand side of \eqref{eq:gen_aug_sys} in the proof.  We now define the candidate ISS Lyapunov function $W$ as follows, for any $\chi=(x,e) \in \mathbb{R}^{n_x+n_p}$ and some $\lambda>0$:
\begin{align}
W(\chi)=V(x)+\lambda \sigma(\Vert e \Vert)\label{eq:def_lya_aug}
\end{align}
We first introduce a technical assumption on $\sigma$.
 \begin{assumption}\label{as:homogeneous}
Function $\sigma$ is $\mathcal{K}_{\infty}$ and $s$-homogeneous for some $s>0$ meaning that for any $\lambda_1>0$ and $r\geq0$, $\sigma(\lambda_1 r)=\lambda_1^s\sigma(r)$.
\end{assumption}
 Note that under Assumption \ref{as:homogeneous}, for any $\chi\in \mathbb{R}^{n_x+n_p}$:
 \begin{align}
     \underline{\alpha}_1(\Vert \chi \Vert)\leq W(\chi) &\leq  \bar{\alpha}_1(\Vert \chi \Vert),\label{eq:lya_bound_aug}
 \end{align}
  where for any $r\geq 0$, $\underline{\alpha}_1(r)=\min(\underline{\alpha},\lambda \sigma)(\frac{1}{2}r)$ and $\bar{\alpha}_1(r)=\max(\bar{\alpha},\lambda \sigma)(r)$. Note that $\underline{\alpha}_1$ and $\bar{\alpha}_1$ are $\mathcal{K}_{\infty}$-functions as $\underline{\alpha}$, $\bar{\alpha}$ and $\sigma$ are  $\mathcal{K}_{\infty}$-functions.
\begin{assumption}\label{as:initial_observability}
 There exists $M_0$ and $\bar{\delta}_0$ such that $\Vert x_{L-1,L}\Vert \leq M_0$ and System \eqref{eq:gen_dyn_sys} is \emph{L-step observable after feedback} on $\widebar{B}(0,M_0)$ at time $L-1$ with level $\bar{\delta}_0$.
 
\end{assumption}
 For any $R>0$, we set $\Omega_R=\{\chi \in \mathbb{R}^{n_x+n_p}| W(\chi)\leq R\}$ 
We can now state the main result of the section.
\begin{theorem}\label{th:ISS_aug_sys}
   Under Assumptions \ref{as:regularity+compact_control_set}, \ref{as:bounded_noise},  \ref{as:penalisation}, \ref{as:opti_algo}, \ref{as:MPC_admissiblity} and \ref{as:initial_observability}, there exist $M_0>0$, $\bar{\delta}_0>0$ $\widebar{R}>0$, $\bar{\delta}>0$, $r_0>0$, $s>0$, ${r}_{max}>0$, $\lambda>0$, $\nu_{\max}>0$, $0<\gamma<\frac{1}{2}$, a $\mathcal{K}_\infty$-function $\alpha_1$ and a $\mathcal{K}$-function $\sigma_1$ such that if the following are satisfied: 
   \begin{align}
    \nu &\leq \nu_{\max},& r_0 &\leq r_{\max}, &K(M_0,\nu)\nu<\min(\bar{\delta}_0,\bar{\delta}),\label{eq:condition_parameters}
    \end{align}
   with  $K(M_0,\nu)$ from Proposition \ref{prop:existence_uniqueness} and if $\Vert x_{L-1,L} \Vert\leq M_0 $, and $\chi_{L-1}=(x_{L-1},e_{L-1})\in \Omega_{\widebar{R}}\cap (\mathbb{R}^{n_x}\times \widebar{B}(0,r_0))$  then for any $t\geq L-1$, $\hat{u}_t$ and  $p_t^*$ are well defined and the following hold:
   \begin{align}
       &\Vert x_{t,L} \Vert\leq M_0,&  &\Vert p_t^*-\bar{p}\Vert\leq \nu,& &\chi_t\in \Omega_{\widebar{R}}&
      &\Vert e_t \Vert\leq r_0,&\label{eq:th_bound_error}
    \end{align}
   \begin{align}
    \Vert e_{t+1}\Vert &\leq \gamma \Vert e_t\Vert + (1+\gamma)\Vert p_{t+1}^* -\bar{p}\Vert \label{eq:th_bound_dyn_error}\\
     W(\chi_{t+1})&\leq W(\chi_t)-\alpha_1(\Vert \chi_t \Vert)+\sigma_1(\Vert w_t\Vert).\label{eq:th_lya_dissipative_aug}
   \end{align}
   \begin{proof}
   For the sake of conciseness, we only provide a sketch of proof. It is made by strong induction. The initialisation is ensured by Assumption \ref{as:initial_observability}. Then, by assuming that \eqref{eq:th_bound_error}-\eqref{eq:th_lya_dissipative_aug} hold for any $L\leq k \leq t$, one can prove that $\hat{u}_{t+1}$ and  $p_{t+1}^*$ are well defined by invoking Proposition \eqref{prop:existence_uniqueness}, Assumption \ref{as:MPC_admissiblity} and Constraint \eqref{eq:con3}. One further gets \eqref{eq:th_bound_dyn_error} from \eqref{eq:opti_algo_convergence} and \eqref{eq:online_estimate} by the triangle inequality. Finally, from Assumption \ref{as:ISS_nominal_state_feedback} and Constraint \eqref{eq:con2} and several manipulation of $\mathcal{K}$-functions, one gets \eqref{eq:th_lya_dissipative_aug} which leads to \eqref{eq:th_bound_error} at time $t+1$ thanks to the strong induction hypothesis and Remark 3.7 of \cite{jiang_input--state_2001}.
  
   \end{proof}
 
\end{theorem}

One can further notice  from the non-decreasing property of $K$ that the set of conditions \eqref{eq:condition_parameters} is feasible uniformly with respect to  $t$ for $\nu$ and $r_0$ sufficiently small.

\begin{corollary}\label{cor:ISS_aug_sys}
Under the assumptions and the settings of Theorem \ref{th:ISS_aug_sys}, there exist $\widebar{R}>0$ and $r_0>0$ such that  system  
 \eqref{eq:gen_aug_sys} is ISS on $\Omega_{\widebar{R}}\cap(\mathbb{R}^{n_x}\times\widebar{B}(0,r_0))$. Moreover,  
 \begin{align}
 \limsup_{t\rightarrow+\infty}{\Vert e_t \Vert}\leq \frac{1+\gamma}{1-\gamma}\nu.   \label{eq:lim_error}  
 \end{align}
 
 \end{corollary}
 \begin{proof}
 From Equations \eqref{eq:th_bound_error} and \eqref{eq:th_bound_dyn_error}, it is clear that $\Omega_{\widebar{R}}\cap(\mathbb{R}^{n_x}\times\widebar{B}(0,r_0))$ is a RPI set. The ISS property follows from \eqref{eq:lya_bound_aug}, \eqref{eq:th_lya_dissipative_aug} and Proposition \ref{prop:ISS}. Equation \eqref{eq:lim_error} is a direct consequence of Equation \eqref{eq:th_bound_dyn_error}.
 \end{proof}

Theorem \ref{th:ISS_aug_sys} states that both the estimation error and the state of system \eqref{eq:gen_dyn_sys} are ultimately bounded by the magnitude of the measurement noise $\nu$.  Corollary \ref{cor:ISS_aug_sys} expresses this fact in terms of a Regional ISS property of system \eqref{eq:gen_aug_sys}. It also gives a more accurate ultimate bound on $e_t$  

\section{APPLICATION: BEARING-ONLY ACTIVE SLAM}\label{section:application}

In this section we consider a 2D robot represented by position variables $x=(x_1,x_2)$ and following a first order dynamics defined for any $t\in \mathbb{N}$ by:
\begin{align}
    x_{t+1}=x_t+\Delta u_t,\label{eq:dyn_simple_int}
\end{align}
where $u_t$ is a velocity input and $\Delta>0$. The unknown parameter in this context is the 2D position of a landmark $\bar{p}=(\bar{p}_1,\bar{p}_2)$. The state $x$ and the parameter $\bar{p}$ are supposed to be observed through noisy bearing angle measurements, denoted by $y$ and which can be written as follows for any $t\geq0$:
\begin{align}
    y_t=h(x_t,\bar{p})+v_t=\frac{\bar{p}-x_t}{\Vert \bar{p}-x_t\Vert}+v_t,\label{eq:obs_bearing}
\end{align}
with $\Vert v_t\Vert \leq \nu$. In this section, we focus on the sensor-centric view of Simultaneous Localisation and Mapping (SLAM)  in which the state of the system $x_t$ is supposed to be already estimated up to some level of precision. Therefore, one can assume, as in section \ref{section_setup}, that a noisy measurement of $x_t$, denoted by $x^0_t$, is available so that $ x^0_t=x_t+v^0_t$ with $\Vert v^0_t \Vert\leq \nu$. The estimation algorithm presented in Problem \eqref{eq:noisy_MHE} and Equation \eqref{eq:online_estimate} can be thought of as an online version of a Graph SLAM, see \cite{thrun_graph_2006}.   With this in mind, the controller \eqref{eq:total_control} can be seen as an Active SLAM controller aiming to ensure the good quality of  Algorithm \eqref{eq:online_estimate}  through the resolution of Problem \eqref{eq:MPC_obs} and especially Constraint \eqref{eq:con3}, see \cite{chen_active_2011} for more details. The controller $\kappa$ has be chosen as a smoothly saturated linear input and the Observability Grammian $\mathcal{O}$ can be written as follows for any $ t\geq L-1$:
\[     \mathcal{O}(\bar{p},x_{t+1,L},u_{t+1,L})=\sum_{k=t-L+1}^t H(x_k,\bar{p})H^T(x_k,\bar{p}),\]
where:
\begin{align*}
    H(x,\bar{p}) =\frac{1}{\Vert x-\bar{p}\Vert^2}\begin{bmatrix}x_2-\bar{p}_2\\ -(x_1-\bar{p}_1) \end{bmatrix}.
\end{align*}
The simulation has been carried out in MATLAB using the PENLAB toolbox for nonlinear Programming and Semi-Definite Programming to solve Problem \eqref{eq:noisy_MHE} and \eqref{eq:MPC_obs}, see \cite{fiala_penlab_2013}. Note that $\psi_t$ in this context represents $10$ iterations of the optimisation routine. Figure \ref{fig:traj} represents a horizontal trajectory where only the controller $\kappa$ has been applied and a trajectory resulting from \eqref{eq:total_control}. Figure \ref{fig:error} represents the corresponding estimation errors in the landmark position. Both the trajectories have been simulated with the following choice of parameters: $\bar{p}=(5,8)$, $\Delta=0.1$, $L=10$, $\nu=0.03$, $\delta'=1$, $\mu=0.5$, $c=0$, $\hat{p}_{L-1}=(3,10)$, $x_0=(5,10)$. One can see that in this case the horizontal trajectory does not allow the proper resolution of the MHPE problem and leads to a diverging estimate. On the contrary, the observability-seeking trajectory exhibits piece-wise circular behaviours which are known to ensure observability \cite{flayac_non-uniform_nodate}, \cite{shames_circumnavigation_2012} and leads to good estimation performance.

\begin{figure}[h]
\begin{center}
\includegraphics[scale=0.63]{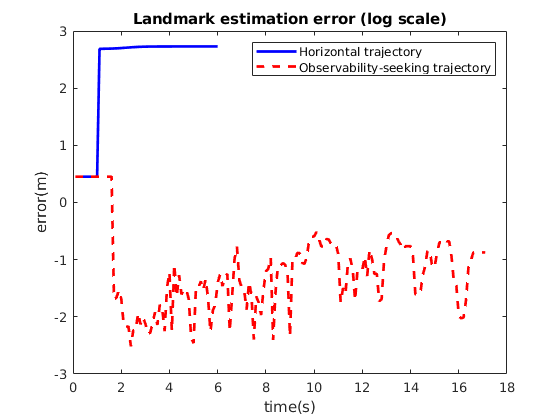}
\end{center}
\caption{Plot of the estimation error in the landmarks position for a horizontal trajectory and a observability-seeking trajectory   }
\label{fig:error}
\end{figure}

\begin{figure}[h]
\begin{center}
\includegraphics[scale=0.65]{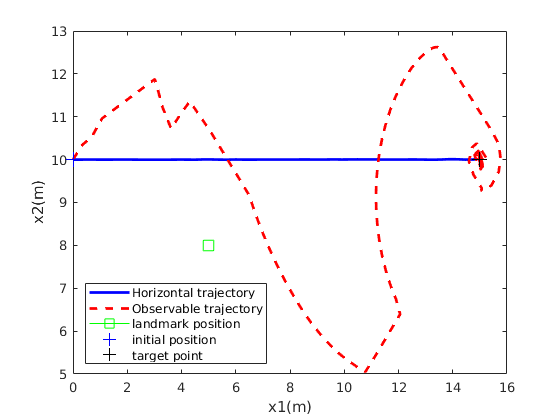}
\end{center}
\caption{Comparison between a horizontal trajectory and an observability-seeking trajectory obtained from the controller \eqref{eq:total_control} }
\label{fig:traj}
\end{figure}
\section{CONCLUSION}
In this paper,  an output-feedback algorithms for adaptive control based on a fast MHPE scheme and an MPC with a constraint on the Observability Grammian has proposed. The closed-loop stability of the system has been proved in terms of input-to-state stability of the augmented system composed of the original state and the parameter estimation error. The method has been numerically tested and validated on the nonlinear application of Active SLAM.



\bibliographystyle{plain}
\bibliography{bibfile}

\end{document}